\documentclass{compositio}

\usepackage{amsmath,amsthm,amssymb, comment}
\usepackage[mathscr]{eucal}
\usepackage{enumerate}
\usepackage{ulem}
\usepackage[usenames]{color}

\newtheorem{theorem}{Theorem}
\newtheorem{proposition}[theorem]{Proposition}

\newtheorem{lemma}[theorem]{Lemma}
\newtheorem{corollary}[theorem]{Corollary}

\newtheorem{remark}[theorem]{Remark}

\newcommand{\sm}[4]{\left(\begin{smallmatrix}#1&#2\\ #3&#4 \end{smallmatrix}
\right)}

\newcommand{\C}{\mathbb{C}}

\newcommand{\Q}{\mathbb{Q}}

\newcommand{\Z}{\mathbb{Z}}

\newcommand{\SL}{{\text {\rm SL}}}

\def\H{\mathbb{H}}

\def\f{f_{N,k}}
\def\D{\Delta_{N,k}(\tau)}

\title{Modular parametrizations of certain elliptic curves}
\author{Matija Kazalicki}
\email{mkazal@math.hr}
\address{Department of Mathematics, University of Zagreb, Bijenicka cesta 30, Zagreb, Croatia}

\author{Yuichi Sakai}
\email{ dynamixaxs@gmail.com}
\address{}

\author{Koji Tasaka}
\email{ k-tasaka@math.kyushu-u.ac.jp}
\address{ Graduate school of mathematics, Kyushu University, 744 Motooka Nishiku, Fukuoka-city, Fukuoka, 819-0395 Japan}


\date{}
\begin{document}


\begin{abstract}

Kaneko and Sakai \cite{KS} recently observed that certain elliptic curves whose associated newforms (by the modularity theorem) are given by the eta-quotients can be characterized by a particular differential equation involving modular forms and Ramanujan-Serre differential operator. 

In this paper, we study certain properties of modular parametrization associated to the elliptic curves over $\Q$, and as a consequence we generalize and explain some of their findings.

\end{abstract}
\maketitle
\section{Introduction}

By the modularity theorem \cite{BCDT,DS}, an elliptic curve $E$ over $\Q$ admits a modular parametrization $\Phi_E:X_0(N)\rightarrow E$ for some integer $N$. If $N$ is the smallest such integer, then it is equal to the conductor of $E$ and the pullback of the N\'eron differential of $E$ under $\Phi_E$ is a rational multiple of $2\pi i f_E(\tau)$, where $f_E(\tau)\in S_2(\Gamma_0(N))$ is a newform with rational Fourier coefficients. The fact that the $L$-function of $f_E(\tau)$ coincides with the Hasse-Weil zeta function of $E$ (which follows from Eichler-Shimura theory) is central to the proof of Fermat's last theorem, and is related to the Birch and Swinnerton-Dyer conjecture. In addition to this, modular parametrization is used for constructing rational points on elliptic curves, and appears in the Gross-Zagier formula \cite{GZ}.

In this paper, we study some general properties of $\Phi_E$, and as a consequences we explain and generalize the results of Kaneko and Sakai from \cite{KS}. 

Kaneko and Sakai (inspired by the paper of Guerzhoy \cite{G}) observed that certain elliptic curves whose associated newforms (by the modularity theorem) are given by the eta-quotients from the list of Martin and Ono \cite{MO} can be characterized by a particular differential equation involving holomorphic modular forms.

To give an example of this phenomena, let $f_{20}(\tau)=\eta(\tau)^4\eta(5\tau)^4$ be a unique newform of weight $2$ on $\Gamma_0(20)$, where $\eta(\tau)$ is the Dedekind eta function $\eta(\tau)=q^{1/24} \prod_{n>0} (1-q^n)$, $q=e^{2\pi i \tau}$, and put $\Delta_{5,4}(\tau)=f_{20}(\tau/2)^2$.  
Then an Eisenstein series $Q_5(\tau)$ on $M_4(\Gamma_0(5))$ associated either to cusp $i\infty$ or to cusp $0$ is a solution of the following differential equation 
\begin{align}\label{eq:ks}
\partial_{5,4}(Q_5)^2&=Q_5^3-\frac{89}{13}Q_5^2 \Delta_{5,4}-\frac{3500}{169} Q_5 \Delta_{5,4}^2-\frac{125000}{2197} \Delta_{5,4}^3,
\end{align}
 where  $\partial_{5,4}(Q_5(\tau))=\frac{1}{2\pi i}Q_5(\tau)'-\frac{1}{2\pi i}Q_5(\tau)\Delta_{5,4}(\tau)'/\Delta_{5,4}(\tau)$ is a Ramanujan-Serre differential operator. Throughout the paper, we use symbol $'$ to denote $\frac{d}{d\tau}$.
This differential equation defines a parametrization of an elliptic curve $E:y^2=x^3-\frac{89}{13}x^2-\frac{3500}{169}x-\frac{125000}{2197}$  by modular functions
\begin{equation*}
 x=\frac{Q_5(\tau)}{\Delta_{5,4}(\tau)},\quad y=\frac{\partial_{5,4}(Q_5)(\tau)}{\Delta_{5,4}(\tau)^{3/2}}, 
\end{equation*}
 and $f_{20}(\tau)$ is the newform associated to $E$.
One finds that $\Delta_{5,4}(\tau)\in S_4(\Gamma_0(5))$, so curiously the modular forms $\Delta_{5,4}, Q_5$ and $\partial(Q_5)$  appearing in this parametrization are modular for $\Gamma_0(5)$, although the conductor of $E$ is $20$.

Using the Eichler-Shimura theory, we generalize \eqref{eq:ks} to the arbitrary elliptic curve $E$ of conductor $4N$, $E:y^2=x^3+ax^2+bx+c$, where $a,b,c \in \Q$, which admits a modular parametrization $\Phi:X\rightarrow E$ satisfying 
\begin{equation*}
\Phi^*\left(\frac{dx}{2y}\right)=\pi i  f_{4N}(\tau/2)d\tau. 
\end{equation*}
Here $X$ is the modular curve $\H/ {\sm {\frac{1}{2}} 0 0 1}^{-1} \Gamma_0(4N) {\sm {\frac{1}{2}} 0 0 1}$, and $f_{4N}(\tau)\in S_2(\Gamma_0(4N))$ is a newform with rational Fourier coefficients associated to $E$. It follows from the modularity theorem that in any $\Q$-isomorphism class of elliptic curves there is an elliptic curve $E$ admitting such parametrization (note that for $u \in \Q^\times$ the change of variables $x=u^2X$ and $y=u^3Y$ implies $\frac{dX}{Y}=u \frac{dx}{y})$. 

To such $\Phi$ we associate a solution $Q(\tau)=x(\Phi(\tau)) f_{4N}(\tau/2)^2$  of a differential equation 
\begin{align}
\partial_{N,4}(Q)^2&=Q^3+a Q^2 \Delta_{N,4}+bQ \Delta_{N,4}^2+c\Delta_{N,4}^3,
\end{align}
where $\Delta_{N,4}(\tau)=f_{4N}(\tau/2)^2$, and $\partial_{N,4}(Q(\tau))=\frac{1}{2\pi i}Q(\tau)'-\frac{1}{2\pi i}Q(\tau)\Delta_{N,4}(\tau)'/\Delta_{N,4}(\tau)$.

We show in Corollary \ref{prop:delta} that $f_{4N}(\tau/2)^2$ is modular for $\Gamma_0(N)$. In general the solution $Q(\tau)$ will not be holomorphic and will be modular only for ${\sm {\frac{1}{2}} 0 0 1}^{-1} \Gamma_0(4N) {\sm {\frac{1}{2}} 0 0 1}$, but if the preimage of the point at infinity of $E$ under $\Phi$ is contained in cusps of $X$ and is invariant under the action of ${\sm 1 0 N 1}$ and ${\sm 1 1 0 1 }$ (acting on $X$ by M\"obius transformations), $Q(\tau)$ will be both holomorphic and modular for $\Gamma_0(N)$ (for more details see Proposition \ref{prop:main} and Theorem \ref{thm:mod}). Moreover, in Theorem \ref{thm:hol} we show that there are only finitely many (up to isomorphism) elliptic curves $E$ admitting $\Phi$ with these two properties.

We also obtain similar results generalizing the other examples from \cite{KS} that correspond to the elliptic curves over $\Q$ with $j$-invariant $0$ and $1728$ (see the next section).

\section{Main results}

Throughout the paper, let $N$ be a positive integer and $k\in \{4,6,8,12\}$. Let $E_k/\Q$ be an elliptic curve given by the short Weierstrass equation
$y^2=f_k(x)$, where
\begin{equation*}
\begin{array}{lcl}
f_4(x)&=&x^3 + a_2x^2+a_4x+a_6,\\
f_6(x)&=&x^3 + b_6,\\
f_8(x)&=&x^3 + c_4x,\\
f_{12}(x)&=&x^3 +d_6,\\
\end{array}
\end{equation*}
and $a_2, a_4, a_6, b_6, c_4, d_6 \in \Q$. Moreover, we assume $j(E_4)\ne 0, 1728.$

\noindent Let $$\displaystyle\f(\tau) \in S_2\left(\Gamma_0\left(\frac{k^2}{4}N\right)\right)$$ be a newform with rational Fourier coefficients, and let $\Gamma_k:={\sm {\frac{2}{k}} 0 0 1}^{-1} \Gamma_0(\frac{k^2}{4}N) {\sm {\frac{2}{k}} 0 0 1}$. Define $$\D:= \f(2\tau/k)^{k/2}\in S_k(\Gamma_k).$$ 



\noindent For $f(\tau)\in M^{\textrm{mer}}_4(\Gamma_k)$, we define the (Ramanujan-Serre) differential operator by
$$\partial_{N,k}(f(\tau))=\frac{k}{8\pi i} f'(\tau)-\frac{1}{2\pi i}f(\tau)\frac{\Delta_{N,k}'(\tau)}{\Delta_{N,k}(\tau)}\in M_6^{\textrm{mer}}(\Gamma_k).$$

Finally, assume that there is a meromorphic modular form $Q_k(\tau)\in M^{\textrm{mer}}_4(\Gamma_k)$, such that the corresponding differential equation holds

\begin{equation}\label{eq:1}
\begin{array}{rcl}
\partial_{N,4}(Q_4(\tau))^2 &=& Q_4(\tau)^3+a_2Q_4(\tau)^2\Delta_{N,4}(\tau)+a_4 Q_4(\tau)\Delta_{N,4}(\tau)^2+a_6 \Delta_{N,4}(\tau)^3\\
\partial_{N,6}(Q_6(\tau))^2&=& Q_6(\tau)^3+b_6 \Delta_{N,6}(\tau)^2\\
\partial_{N,8}(Q_8(\tau))^2&=& Q_8(\tau)^3+c_4 Q_8(\tau)\Delta_{N,8}(\tau)\\
\partial_{N,12}(Q_{12}(\tau))^2&=& Q_{12}(\tau)^3+d_6 \Delta_{N,12}(\tau).
\end{array}
\end{equation}

Each of these four identities defines a modular parametrization $\Psi_k:X_k \rightarrow E_k$
\begin{equation*}
\Psi_k(\tau)=\left( \frac{Q_k(\tau)}{\D^{4/k}},\frac{\partial_{N,k}(Q_k)(\tau)}{\D^{6/k}} \right),
\end{equation*}
where $X_k$ is the compactified modular curve $\H/\Gamma_k$.

\begin{proposition}\label{prop:1} Let $\frac{dx}{2y}$ be the N\'eron differential on $E_k$. Then
\begin{equation}\label{eq:par}
\Psi_k^*\left(\frac{dx}{2y}\right)=\frac{4\pi i}{k}  \f(2\tau/k)d\tau. 
\end{equation}
In particular, the conductor of $E_k$ is $\frac{k^2}{4}N$ and $\f(\tau)$ is the cusp form associated to $E_k$ by the modularity theorem.
\end{proposition}

\begin{remark}
Note that when $k=6,8 \mbox{ or }12$, $\f(\tau)$ is a modular form with complex multiplication by the ring of integers of $\Q(\sqrt{-3})$, $\Q(\sqrt{-1})$ and $\Q(\sqrt{-3})$ respectively.
\end{remark}

Conversely, given a modular parametrization $\Phi_k:X_k\rightarrow E_k$ satisfying ($\ref{eq:par}$), we construct a differential equation (\ref{eq:1}) and its solution $Q_k(\tau)$ as follows.
 
Let $x$ and $y$ be functions on $E_k$ satisfying Weierstrass equation $y^2=f_k(x)$. Functions $x(\tau):=x\circ\Phi_k(\tau)$ and $y(\tau):=y\circ\Phi_k(\tau)$ satisfy  $y(\tau)^2=f_k(x(\tau))$. Moreover \eqref{eq:par} implies that
\begin{equation}\label{eq:5}
\left(\frac{k}{8\pi i}x'(\tau)\right)^2=\f(2\tau/k)^2 y(\tau)^2=\D^{4/k}f_k(x(\tau)).
\end{equation}
Define $Q_k(\tau) := x(\tau) \D^{4/k}$.

\begin{proposition}\label{lemma:j} The following formula holds
$$\partial_{N,k}(Q_k(\tau))^2=\D^{12/k}f_k(x(\tau)).$$
In particular, $Q_k(\tau)$ is a solution of $(\ref{eq:1})$.
\end{proposition}

Now we investigate conditions under which $Q_k(\tau)$ is holomorphic. The following lemma easily follows from the formula above.
\begin{lemma}
Assume that $\tau_0\in X_k$ is a pole of $x(\tau)$. Then 
\begin{equation*}
ord_{\tau_0}(Q_k(\tau)) = \begin{cases} 0, & \mbox{if } \tau_0\mbox{ is a cusp,} \\ -2, & \mbox{if } \tau_0\in \H. \end{cases}
\end{equation*}
\end{lemma}

As a consequence, we have the following characterization of the holomorphicity of $Q_k(\tau)$ in terms of modular parametrization $\Phi_k$.
Denote by $\mathcal{C}$ the set of cusps of $X_k$, and by $\mathcal{O}$ the point at infinity of $E_k$.
\begin{proposition}\label{prop:main}
We have that
$Q_k(\tau)$ is holomorphic if and only if $\Phi_k^{-1}(\mathcal{O})\subset \mathcal{C}$.  
\end{proposition}

In Section \ref{sec:hol} we show that the degree of $\Phi_k$ (as a function of the conductor) grows faster than the total ramification index at cusps hence the following theorem holds. 

\begin{theorem}\label{thm:hol}
There are finitely many elliptic curves $E/\Q$ (up to a $\Q$-isomorphism) that admit a modular parametrization $\Phi:X_k\rightarrow E$ with the property that $\Phi^{-1}(\mathcal{O}) \subset \mathcal{C}$. 

In particular, there are finitely many elliptic curves $E_k$ (up to a $\Q$-isomorphism) for which $Q_k(\tau)$ (which satisfy equation (\ref{eq:1})) is holomorphic.
\end{theorem}
Define  $A={\sm 1 0 N 1}$ and $T={\sm 1 1 0 1 }$. It is easy to see that $\Gamma_k$ is generated by $\Gamma_0(N)$ and $A$ and $T$ (Lemma \ref{lemma:easy}), hence $Q_k(\tau)$ is modular for $\Gamma_0(N)$ if and only if it is invariant under the action of slash operators $|A$ and $|T$.
The following theorem describes the modularity in terms of parametrization $\Phi_k$.

\begin{theorem}\label{thm:mod}
If $\Phi_k^{-1}(\mathcal{O})$ is invariant under $A$ and $T$, then $Q_k(\tau)$ is modular for $\Gamma_0(N)$. 
\end{theorem}

\section{Proofs}

\subsection{Proof of Proposition \ref{prop:1} and Proposition \ref{lemma:j}}

\begin{proof}[Proof of Proposition \ref{prop:1}]
\begin{equation*}
\begin{array}{lcl}
\Psi_k^*\left(\frac{dx}{2y}\right)&=& \frac{1}{2}\frac{d}{d\tau}\left( \frac{Q_k(\tau)}{\D^{4/k}}\right) \frac{\D^{6/k}}{\partial_{N,k}(Q_k)(\tau)}d\tau\\&=&\frac{1}{2}\frac{\frac{d}{d\tau}Q_k(\tau)\f(2\tau/k)^2-\frac{d}{d\tau}\f(2\tau/k)^2Q_k(\tau)}{\f(2\tau/k)^4}\frac{\f(2\tau/k)^3}{\frac{k}{8\pi i}\frac{d}{d\tau}Q_k(\tau)-Q_ks(\tau)\frac{\frac{d}{d\tau}\f(2\tau/k)^{k/2}}{2\pi i \f(2\tau/k)^{k/2}}}d\tau\\
&=& \frac{4\pi i}{k}  \f(2\tau/k)d\tau.  

\end{array}
\end{equation*}
\end{proof}
\begin{proof}[Proof of Proposition \ref{lemma:j}]
By definition, 
\begin{equation*}
\begin{array}{lcl}
\partial_{N,k}(Q_k(\tau)) &=&\frac{k}{8\pi i}(x(\tau) \D^{4/k})'-\frac{1}{2\pi i}x(\tau) \D^{4/k}\frac{\Delta_{N,k}'(\tau)}{\Delta_{N,k}(\tau)}\\ &=&\frac{k}{8\pi i} x'(\tau)\D^{4/k}.
\end{array}
\end{equation*}
Hence the claim follows from (\ref{eq:5}).
\end{proof}

\subsection{Proof of Theorem \ref{thm:hol}}\label{sec:hol}

Let $e_x\in \Z$ be the ramification index of $\Phi_k$ at $x \in X_k$, and let $\deg(\Phi_k)$ be the degree of $\Phi_k$. It follows from the Hurwitz formula that $\sum_{x\in X_k}(e_x-1)=2g-2$, where $g$ is the genus of $X_k$ (note that the genus of $X_k$ is equal to the genus of $\Gamma_0(\frac{k^2}{4}N)$). Therefore $\Phi_k^{-1}(\mathcal{O})\subset \mathcal{C}$ implies
\begin{equation}\label{eq:6}
\deg(\Phi_k) \le \sum_{x\in \mathcal{C}}e_{x} \le 2g-2+\#\mathcal{C}.
\end{equation}
In \cite{Wa}, Watkins proved a lower bound for the degree of modular  parametrization $\Phi$ of an elliptic curve over $\Q$ of conductor $M$ 
$$\deg(\Phi)\ge \frac{M^{7/6}}{\log{M}}\cdot \frac{1/10300}{\sqrt{0.02+\log\log{M}}}.$$
On the other hand, an upper bound (see \cite{CWZ}) for the genus $g$ of $X_0(M)$ is
$$g<M\frac{e^\gamma}{2\pi^2}(\log\log M+2/\log\log M)\mbox{ for } M>2,$$
where $\gamma=0.5772\ldots$ is Euler's constant.

If we use a trivial bound $\#\mathcal{C}\le M$, an easy calculation shows that (\ref{eq:6}) can not hold for curves $E_k$ of conductor greater than $10^{50}$. Therefore, we have proved the Theorem \ref{thm:hol}.

\begin{remark}
If we assume that ramification index at cusps is bounded by 24 (as suggested in the paper of Brunault \cite{B}), and if we use Abramovich \cite{A} lower bound for modular degree $\deg(\Phi) \ge 7M/1600$, we obtain that (\ref{eq:6}) can not hold for elliptic curves of conductor greater than $2^{19}$.
\end{remark}

\subsection{Proof of Theorem \ref{thm:mod}}

In this section we investigate conditions on modular parametrization $\Phi_k$ under which $\D$ and $Q_k(\tau)$, initially modular for $\Gamma_k$, are modular for $\Gamma_0(N)$.

For $S={\sm a b c d} \in \SL_2(\Z)$, and a (meromorphic) modular form $f(\tau)$ of weight $l$, we define the usual slash operator as $f(\tau)|_lS:=f(S\tau)(c\tau+d)^{-l}$, where $S\tau=\frac{a \tau + b}{c\tau + d}$. Define $T={\sm 1 1 0 1 }$ and $A={\sm 1 0 N 1}$. 

\begin{lemma}\label{lemma:easy}
Group $\Gamma_0(\frac{k}{2}N)$ is generated by $\Gamma_k$ and $T$, while $\Gamma_0(N)$  is generated by $\Gamma_0(\frac{k}{2}N)$ and $A$.
\end{lemma}
\begin{proof}
To prove the first statement, let ${\sm a b c d}\in \Gamma_0(\frac{k}{2}N)$. Then $\gcd(a,\frac{k}{2})=1$, and there is $r\in \Z$ such that $ar\equiv -b \bmod{\frac{k}{2}}$. Then ${\sm a b c d}T^r\in \Gamma_k=\Gamma_0(\frac{k}{2}N)\cap \Gamma^0(\frac{k}{2})$, and the claim follows.

Second statement is proved analogously. 

\end{proof}

Therefore, to prove that $\D$ and $Q_k(\tau)$ are modular for $\Gamma_0(N)$ it suffices to show their invariance under the slash operators $|T$ and $|A$.

\begin{lemma}\label{lemma:normal}
Matrices $A$ and $T$ normalize $\Gamma_k$.
\end{lemma}
\begin{proof}
Let ${\sm a b c d}\in \Gamma_k=\Gamma_0(\frac{k}{2}N)\cap \Gamma^0(\frac{k}{2})$. Then $\frac{k}{2}N | c$ and $\frac{k}{2}|c$, and $ad\equiv 1 \pmod{\frac{k}{2}}$. In particular, since $\frac{k}{2} \in \{2,3,4,6\}$, it follows that $a\equiv d \pmod{\frac{k}{2}}$.

Since
\begin{equation*}
\begin{array}{rcl}
A^{-1}{\sm a b c d} A&=&{\sm {a+bN} b {-aN-bN^2+c+dN} {-bN+d}},\\
T^{-1}{\sm a b c d}T&=&\displaystyle{\sm {a-c} {a+b-c-d} c {c+d}},
\end{array}
\end{equation*}
the claim follows.
\end{proof}

For a prime $p$, define the Hecke operator $T_p$ as a double coset operator $\Gamma_k {\sm 1 0 0 p} \Gamma_k$ acting on the space of cusp forms on $\Gamma_k$. Slash operators $|A$ and $|T$ correspond to $\Gamma_k A \Gamma_k$ and $\Gamma_k T \Gamma_k$ (see Chapter 5 of \cite{DS}).

Define the Fricke involution $|_2B$ on $S_2(\Gamma_k)$ by the matrix $B:={\sm 0 {-\frac{k}{2}} {\frac{k}{2}N} 0 }$. Note that $|_2B$ is the conjugate of the usual Fricke involution on $\Gamma_0(\frac{k^2}{4}N)$. In particular, $B$ normalizes $\Gamma_k$, and $|_2B$ commutes with all the Hecke operators $T_p$, $p \nmid \frac{k^2}{4}N$. Hence, $\f(2\tau/k)|_2B=\lambda_{k,N} \f(2\tau/k)$ for some $\lambda_{k,N} = \pm 1$.

\begin{lemma}\label{lemma:main}
The following are true.
\begin{enumerate}
\item[a)]
$$\f(2\tau/k)|_2T=e^{4\pi i/k}\f(2\tau/k),$$
\item[b)]
$$\f(2\tau/k)|_2A= e^{-4\pi i/k}\f(2\tau/k).$$
\end{enumerate}
In particular, $|_2A$ and $|_2B$ have order $\frac{k}{2}$ when acting on $\f(2\tau/k)$.
\end{lemma}
\begin{proof}
A key observation is that the Fourier coefficients of $\f(\tau)$ are supported at integers that are $1 \bmod \frac{k}{2}$. This implies
$$\f(2\tau/k)|_2T=e^{4\pi i/k}\f(2\tau/k).$$
When $k=4$ (and $k=12$) this is a consequence of the general fact that $a_f(2)=0$ whenever $f(\tau)=\sum a_f(n)q^n$ is a newform of level divisible by $4$ (see \cite{O}, p.29). In the other three cases, $\f(\tau)$ is a modular form with complex multiplication by the ring of integers of $\Q(\sqrt{-3})$ or $\Q(\sqrt{-1})$, hence its Fourier coefficients $a_{\f}(p)$ are zero when $p$ is an inert prime (i.e. $p \equiv 2 \pmod{3}$ or $p \equiv 3\pmod{4}$ respectively). Multiplicativity of the Fourier coefficients then implies the observation.

On the other hand $A= BT^{-1}B^{-1}$, therefore
\begin{equation*}
\begin{array}{lcl}
\f(2\tau/k)|_2A=(\f(2\tau/k)|_2B)|_2T^{-1}|_2B^{-1}&=&(\lambda_{k,N}\f(2\tau/k)|_2T^{-1})|_2B^{-1}\\
&=&\lambda_{k,N}\lambda_{k,N}^{-1} e^{-4\pi i /k}\f(2\tau/k).
\end{array}
\end{equation*}
\end{proof}

\begin{corollary}\label{prop:delta}
We have that 
\begin{enumerate}
\item[a)] $\D\in S_k(\Gamma_0(N)),$
\item[b)] $\Delta_{N,8}(\tau)^{1/2}|_4A=-\Delta_{N,8}(\tau)^{1/2}$ and $\Delta_{N,8}(\tau)^{1/2}|_4T=-\Delta_{N,8}(\tau)^{1/2},$
\item[c)] $\Delta_{N,12}(\tau)^{1/2}|_6A=-\Delta_{N,12}(\tau)^{1/2}$ and $\Delta_{N,12}(\tau)^{1/2}|_6T=-\Delta_{N,12}(\tau)^{1/2}.$
\end{enumerate}
\end{corollary}

We now recall some basic facts about Jacobians of modular curves. For more details see Chapter 6 of \cite{DS}.  Denote by $Jac(X_k)$ the Jacobian of $X_k$. We will view it either as $S_2(\Gamma_k)^\wedge/H_1(X_k,\Z)$ (where $\gamma\in H_1(X_k,\Z)$ acts on $f(\tau) \in S_2(\Gamma_k)$ by $f(\tau)\mapsto \int_\gamma f(\tau)d\tau$), or as the Picard group $Pic^0(X_k)$ of $X_k$, which is the quotient $Div^0(X_k)/Div^l(X_k)$ of the degree zero divisors of $X_k$ modulo principal divisors. If $x_0$ is a base point in $X_k$ then $X_k$ embeds into its Picard group under the Abel-Jacobi map
$$X_k \rightarrow Pic^0(X_k), \qquad x \mapsto (x)-(x_0),$$
where $(x)-(x_0)$ denotes the equivalence class of divisors $(x)-(x_0)+Div^l(X_k)$. 

It is known that the parametrization $\Phi_k:X_k\rightarrow E_k$ can be factored as
\begin{equation}\label{eq:factor}
X_k \hookrightarrow Jac(X_k) \xrightarrow{\psi_k} \tilde{E_k} \xrightarrow{\phi_k}E_k.
\end{equation}

Here $X_k \hookrightarrow Jac(X_k)$ is the Abel-Jacobi map (for some base point $x_0 \in X_k$), $\phi_k$ is a rational isogeny, and $\tilde{E_k}$ (together with $\psi_k$) is the strong Weil curve associated to the newform $\f(2\tau/k)$ via Eichler-Shimura construction as follows.

Let $V_k$ be a $\C$-span of $\f(2\tau/k)\in S_2(\Gamma_k)$, and define $\Lambda_k:=H_1(X_k)|V_k$. Restriction to $V_k$ gives a homomorphism $\psi_k$
\begin{equation*}
Jac(X_k)\rightarrow V_k^\wedge/\Lambda_k\cong \tilde{E_k}.
\end{equation*}
Here $V_k^\wedge/\Lambda_k$ is a one-dimensional complex torus isomorphic to the rational elliptic curve $\tilde{E}_k$ with the Weierstrass equation $\tilde{E_k}:y^2=x^3-\frac{g_2(\Lambda_k)}{4}x-\frac{g_3(\Lambda_k)}{4}$.

Let $S$ be either $A$ or $T$. Since by Lemma \ref{lemma:normal} $S$ normalizes $\Gamma_k$, we can define the action of $S$ on $Jac(X_k)$ in two equivalent ways: for $\phi\in S_2(\Gamma_k)^\wedge/H_1(X_k,\Z)$ and $f(\tau)\in S_2(\Gamma_k)$ let $S(\phi)(f(\tau)):=\phi(f(\tau)|_2S),$ or for $P=(x)-(x_0)\in Pic^0(X_k)$ let $S(P)=(Sx)-(Sx_0)$. Now Lemma \ref{lemma:main} implies that the action of $S$ on $Jac(X_k)$ descends to the automorphism of $\tilde{E}_k$ of the order $\frac{k}{2}$.

Recall that $x$ and $y$ are functions on $E_k$ satisfying Weierstrass equation $y^2=f_k(x)$, and that $x(\tau)=x\circ\Phi_k(\tau)$ and $y(\tau)=y\circ\Phi_k(\tau)$ are modular functions on $X_k$.

\begin{proposition}\label{prop:xy}
Let $S$ be either $A$ or $T$. If $\Phi_k^{-1}(\mathcal{O})$ is invariant under $A$ and $T$, then 
\begin{enumerate}
\item[a)] 
\begin{equation*}
x(\tau)|S= \begin{cases} x(\tau), & \mbox{if } k=4, \\ -x(\tau), & \mbox{if } k=8.\end{cases}
\end{equation*}
\item[b)]
\begin{equation*}
y(\tau)|S= \begin{cases} y(\tau), & \mbox{if } k=6 ,\\ -y(\tau), & \mbox{if } k=12, 
\end{cases}
\end{equation*}
\end{enumerate}
\end{proposition}

\begin{proof}
For $P\in E_k$, we define the $S(P):=\phi_k(S(\tilde{P}))$ for any $\tilde{P}\in \phi_k^{-1}(P)$.  It is well defined since $S$-invariance of $\Phi_k^{-1}(\mathcal{O})$ implies the $S$-invariance of $Ker(\phi_k)$. We have that $\phi_k(S(P))=S(\phi_k(P))$, hence $S$ is an automorphism of $E_k$.

Let $x_0$ be a base point of Abel-Jacobi map in (\ref{eq:factor}). Then $x_0\in  \Phi_k^{-1}(\mathcal{O})$, hence $\phi_k\circ \psi_k$ maps $(Sx_0)-(x_0)$ to $\mathcal{O}$ in $E_k$. In particular, for $x\in X_k$ we have \begin{equation}\label{eq:S}
\Phi_k(Sx)=\phi_k\circ\psi_k((Sx)-(x_0))=\phi_k\circ\psi_k((Sx)-(Sx_0))=S(\Phi_k(x)).
\end{equation}

Assume first that $k=4$. Then $j(E_4) \ne 0, 1728$, and the automorphism group of $E_4$ is of order $2$ generated by $(x,y)\mapsto (x,-y)$. In particular $x(S(P))=x(P)$, for every $P\in E_4$. 

If $k=8$, then $S$ is an automorphism of order $\frac{k}{2}=4$ of $\tilde{E_k}$, hence $j(\tilde{E_k})=1728$, and $g_3(\Lambda_8)=0$. Moreover $\phi_k$ is isomorphism (defined over $\Q$), which implies that $S$ is an isomorphism of order $4$ of $E_8$ as well. The automorphism group is generated by $(x,y)\mapsto (-x,iy)$, hence $x(S(P))=-x(P)$ for every $P\in E_8$.

If $k=6$ or $12$, then $j(\tilde{E_k})=0$, $g_2(\Lambda_k)=0$, and $\phi_k$ is an isomorphism (defined over $\Q$). Therefore, $S$ has order $3$ on $E_k$ if $k=6$, and order $6$ if $k=12$. The automorphism group is generated by $(x,y)\mapsto (e^{2\pi i /3} x,-y)$, and in particular $y(S(P))=y(P)$ if $k=6$, and $y(S(P))=-y(P)$ if $k=12$, for every $P\in E_k$.

Now (\ref{eq:S}) implies
$$x(\tau)|S=x(S\tau)=x(\Phi_k(S\tau))=x(S(\Phi_k(\tau)))\quad \mbox{and}\quad y(\tau)|S=y(S\tau)=y(\Phi_k(S\tau))=y(S(\Phi_k(\tau))),$$
and the claim follows from the previous paragraph.

\end{proof}

\noindent We need the following technical lemma. Recall that $Q_k(\tau) := x(\tau) \D^{4/k}$.

\begin{lemma}\label{lemma:ramanujan}
If $\partial_{N,k}(Q_k(\tau))\in M^{\textrm{mer}}_6(\Gamma_0(N))$, then $Q_k(\tau)\in M^{\textrm{mer}}_4(\Gamma_0(N))$.
\end{lemma}
\begin{proof}
As in the proof of Proposition \ref{lemma:j}, we have that $\partial_{N,k}(Q_k(\tau))=\frac{k}{8\pi i}x'(\tau)\D^{4/k}=\frac{k}{8\pi i}\frac{x'(\tau)}{x(\tau)}Q_k(\tau).$
Let $S$ be either $A$ or $T$. Then $(x(S\tau))'=x'(\tau)|_2S$, and the invariance of $\frac{x'(\tau)}{x(\tau)}$ under $S$ (hence under $\Gamma_0(N)$) follows from the fact that $x(\tau)$ is an eigenfunction for $S$, which follows from the proof of Proposition \ref{prop:xy}.
\end{proof}

Since $Q_k(\tau) := x(\tau) \D^{4/k}$, the Theorem \ref{thm:mod} for $k=4 \mbox{ and }8$ now follows from a) and b) of Corollary \ref{prop:delta}  and a)  of Proposition \ref{prop:xy}, while $k=6 \mbox{ and } 12$ case follows from $\partial_{N,k}(Q_k)(\tau)=y(\tau)\D^{6/k}$ together with a) and c) of Corollary \ref{prop:delta}, b) of Proposition \ref{prop:xy} and Lemma \ref{lemma:ramanujan}.

\section{Example}

Let $$f_{19,4}(\tau)=\sum_{n=1}^\infty a(n)q^n=q+2q^3-q^5-3q^7+q^9+\cdots$$ be a unique newform in $S_2(\Gamma_0(76))$, and denote by $\Delta_{19,4}(\tau)=f_{19,4}(\tau/2)^2\in S_4(\Gamma_0(19))$.

Set $\Gamma={\sm {\frac{1}{2}} 0 0 1}^{-1} \Gamma_0(76) {\sm {\frac{1}{2}} 0 0 1}$. For $\tau \in \bar{\H}$ we define 

$$ \Psi(\tau)=\pi i\int_{i\infty}^\tau f(z/2)dz.$$ 

 For $\gamma \in \Gamma$ and $\tau \in \bar{\H}$ , define $\omega(\gamma):=\Psi(\gamma\tau)-\Psi(\tau)$. One easily checks that $\frac{d}{d\tau}\omega(\tau)=0$, hence $\omega(\gamma)$ does not depend on $\tau$. Denote by $\Lambda$ the image of $\Gamma$ under $\omega$. By Eichler-Shimura theory $\Lambda$ is a lattice, and $\Psi(\tau)$ induces a parametrization $X:=\H / \Gamma \rightarrow \C / \Lambda$. The complex torus $\C/\Lambda$ is isomorphic to $E: y^2=x^3-\frac{g_2(\Lambda)}{4}x-\frac{g_3(\Lambda)}{4}$ by the map given by Weierstrass $\wp$-function and its derivative, $z \longmapsto \left(\wp(z,\Lambda),\wp'(z,\Lambda)/2\right)$, thus by composing these two maps we obtain a modular parametrization $\Phi:X\rightarrow E$.
 
One finds that generators $\omega_1$ and $\omega_2$ of $\Lambda$ are
$$\omega_1= 1.1104197465122\ldots, \quad \omega_2=0.5552098732561\ldots +
2.1752061725591\ldots \times i.$$ Moreover, $g_2(\Lambda)= \frac{256}{3}$ and $g_3(\Lambda)=\frac{4112}{27}$, hence it follows from Proposition \ref{lemma:j} that 
$$Q(\tau)=\Delta_{19,4}(\tau)\wp(\Psi(\tau),\Lambda)=1+\frac{1}{3}\left( 8q+8q^2+64q^3+232q^4+336q^5+256q^6+512q^7+\cdots\right)$$ satisfies a differential equation
\begin{align}
\partial_{19,4}(Q)^2&=Q^3-\frac{64}{3} Q \Delta_{19,4}^2-\frac{1028}{27}\Delta_{19,4}^3.
\end{align}

One finds that 
$$GCD\left(\left\{p+1-a(p):p \mbox{ prime}, p\equiv 1\pmod{76} \right\}\right)=1,$$
hence it follows from the special case of Drinfeld-Manin theorem (see Theorem 2.20 in \cite{D}) that $\Psi(\tau)$ maps cusps of $X$ to the lattice $\Lambda$, or equivalently that $\Phi$ maps cusps of $X$ to the point at infinity of $E$. Modular curve $X$ has six cusps, and one can check (for example by using software package Magma) that the degree of $\Phi$ is six, therefore the conditions of Proposition \ref{prop:main} and Theorem \ref{thm:mod} are satisfied, and we conclude that $Q(\tau)\in M_4(\Gamma_0(19))$.

\end{document}